\documentclass[11pt]{amsart}
\usepackage{amssymb}
\usepackage{amsmath}
\usepackage{amsfonts}
\usepackage{amsthm}
\usepackage{amssymb}
\usepackage[dvipdfm]{pict2e}

\title[Galois Lines]{Galois lines for space elliptic curve with $j=12^3$}
\author[Mitsunori Kanazawa and Hisao Yoshihara]{}

\newtheorem{theorem}{Theorem}
\newtheorem{lemma}[theorem]{Lemma}

\newtheorem{corollary}[theorem]{Corollary}

\theoremstyle{definition}

\newtheorem{example}[theorem]{Example}
\newcommand{\ds}{\displaystyle}

\theoremstyle{remark}
\newtheorem{remark}[theorem]{Remark}

\begin{document}
\maketitle

\begin{center}

{\sc Mitsunori Kanazawa and Hisao Yoshihara}\\
\medskip
{\small{\em Graduate School of Science and Technology,\\
 Niigata University Niigata 950-2181, Japan and \\
 Department of Mathematics, Faculty of Science, \\ 
Niigata University, Niigata 950-2181, Japan \\ 
E-mail \tt{  kana@js3.so-net.ne.jp and yosihara@math.sc.niigata-u.ac.jp
  }} }
\end{center}

\medskip

\begin{abstract}
The $V_4$-lines for each linearly normal space elliptic curve form the edges of a tetrahedron, however in case the elliptic curve has $j=12^3$, there exist $Z_4$-lines. We show the arrangement of $V_4$ and $Z_4$-lines explicitly for such a curve. As a corollary we obtain that each irreducible quartic curve with genus one has at most two Galois points.   
\end{abstract}

\bigskip

\section{Introduction}
We have been studying Galois embedding of algebraic varieties \cite{y6}, in particular, of  elliptic curves $E$. 
In this case, by Lemma 8 in \cite{y1} we can assume the embedding is associated with the complete linear system $|nP_0|$ for some $n \geq 3$, where $P_0 \in E$. 
Let $f_n : E \hookrightarrow \mathbb P^{n-1}$ be the embedding and put $C_n=f_n(E)$. Then we consider the Galois subspaces, Galois group, the arrangement of Galois subspaces and etc. for $C_n$ in $\mathbb P^{n-1}$.
In the previous papers \cite{dy, y1} we have treated in the case where $n=4$ and settled almost all questions. However, the arrangement of $V_4$ and $Z_4$-lines has not been 
determined in sufficient detail for $j(E)=12^3$, i.e., the curve with an automorphism of order four with a fixed point. In this article we will    complete it. It needs long and tedious computations to determine the $Z_4$-lines explicitly.  
As a byproduct we obtain the number of Galois points for an irreducible quartic curve of genus one, which is a correction of the assertion of Corollary 2 in \cite{y1}. 

\bigskip

\newpage

\section{Statement of result}
\begin{theorem}\label{2}
The arrangement of all the Galois lines for $C_4$, where $j(C_4)=12^3$, is illustrated by the union of the following two figures: 

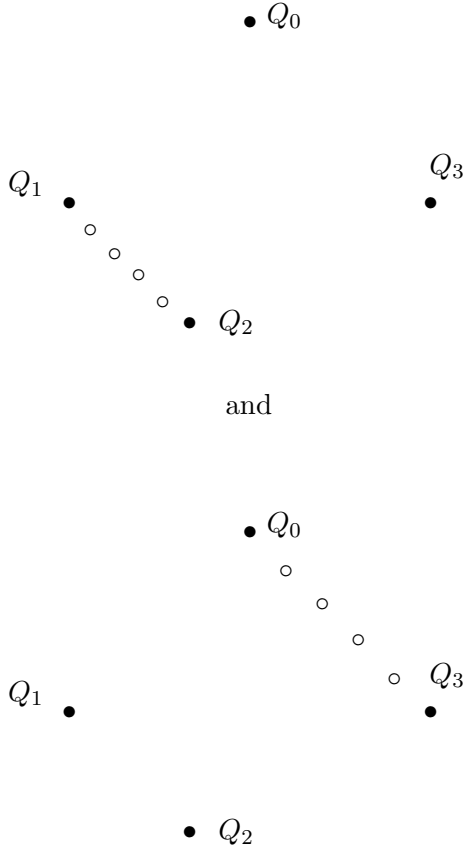
\begin{center}
\setlength{\unitlength}{0.8mm}
\begin{picture}(120,70)
\put(49,4){$\bullet$}
\put(89,24){$\bullet$}
\put(29,24){$\bullet$}
\put(59,54){$\bullet$}
\put(32.5,19.5){$\circ$}
\put(36.5,15.5){$\circ$}
\put(40.5,12){$\circ$}
\put(44.5,7.5){$\circ$}
\put(63,55){$Q_0$}
\put(20,27){$Q_1$}
\put(55,4){$Q_2$}
\put(90,30){$Q_3$}
\put(60,55){\line(-1,-1){33}}
\put(60,55){\line(1,1){5}}
\put(60,55){\line(-1,-5){11}}
\put(60,55){\line(1,5){1}}
\put(60,55){\line(1,-1){33}}
\put(60,55){\line(-1,1){5}}
\put(30,25){\line(1,-1){23}}
\put(30,25){\line(-1,1){5}}
\put(34,21){\line(13,17){30}}
\put(38,17){\line(13,2){4}}
\put(46,18){\line(13,2){6}}
\put(55,19){\line(35,6){37}}
\put(42,13){\line(3,7){20}}
\put(46,9){\line(11,4){3}}
\put(53,11){\line(37,14){39}}
\put(34,21){\line(-13,-17){2}}
\put(38,17){\line(-13,-2){3}}
\put(42,13){\line(-3,-7){2}}
\put(46,9){\line(-11,-4){3}}
\put(50,5){\line(2,1){43}}
\put(50,5){\line(-2,-1){5}}
\put(30,25){\line(1,0){5}}
\put(39,25){\line(1,0){6}}
\put(49,25){\line(1,0){4}}
\put(56,25){\line(1,0){38}}
\put(30,25){\line(-1,0){5}}
\put(90,25){\line(-1,0){32}}
\put(90,25){\line(1,0){5}}
\end{picture}
\end{center}

\begin{center}
{\rm and} 
\end{center}

\begin{center}
\setlength{\unitlength}{0.8mm}
\begin{picture}(120,70)
\put(49,4){$\bullet$}
\put(89,24){$\bullet$}
\put(29,24){$\bullet$}
\put(59,54){$\bullet$}
\put(65,47.5){$\circ$}
\put(71,42){$\circ$}
\put(77,36){$\circ$}
\put(83,29.5){$\circ$}
\put(63,55){$Q_0$}
\put(20,27){$Q_1$}
\put(55,4){$Q_2$}
\put(90,30){$Q_3$}
\put(60,55){\line(-1,-1){33}}
\put(60,55){\line(1,1){5}}
\put(60,55){\line(-1,-5){11}}
\put(60,55){\line(1,5){1}}
\put(60,55){\line(1,-1){33}}
\put(60,55){\line(-1,1){5}}
\put(30,25){\line(1,-1){23}}
\put(30,25){\line(-1,1){5}}
\put(50,5){\line(2,1){43}}
\put(50,5){\line(-2,-1){5}}
\put(30,25){\line(3,2){26}}
\put(30,25){\line(-3,-2){3}}
\put(30,25){\line(-4,-1){3}}
\put(30,25){\line(4,1){24}}
\put(30,25){\line(1,0){23}}
\put(56,25){\line(1,0){4}}
\put(64,25){\line(1,0){10}}
\put(78,25){\line(1,0){16}}
\put(30,25){\line(-1,0){5}}
\put(60,45){\line(3,2){9}}
\put(57,32){\line(4,1){8}}
\put(69,35){\line(4,1){13}}
\put(50,5){\line(11,19){24}}
\put(50,5){\line(17,13){36}}
\put(50,5){\line(-11,-19){3}}
\put(50,5){\line(-17,-13){3}}
\put(30,25){\line(-1,0){5}}
\put(30,25){\line(-1,0){5}}
\put(30,25){\line(-1,0){5}}
\put(30,25){\line(-1,0){5}}

\end{picture}
\end{center}
In these figures, $\bullet$ denotes the intersection of $V_4$-lines and $\circ$ denotes the intersection of a $V_4$ and a $Z_4$-line. Four points $Q_0, \ Q_1 \ Q_2$ and $Q_3$ are not coplanar. These points form vertices of a tetrahedron. Let $\ell_{ij}$ be the line passing through $Q_i$ and $Q_j$ $(0 \le i < j \le 3)$. Then, all the $V_4$-lines are $\ell_{01}, \ \ell_{02}, \ \ell_{03},  \ \ell_{12}, \ \ell_{13} $ and $\ell_{23} $.  Except these lines, each line is a $Z_4$-line. 
For each vertex there exist two $Z_4$-lines passing through it. Two $Z_4$-lines which do not pass through the same vertex are disjoint. A $Z_4$-line meets $V_4$-lines at two points as is shown above. If the one is the vertex $Q_i$, then we let the other be $R_{ij}$, where $0 \le i \le 3$ and $j=1,\ 2$.  
By taking a suitable coordinates of $\mathbb P^3$, we can give the coordinates of $Q_i$ and $R_{ij}$ explicitly as follows, in the following we use the notation $i=\sqrt{-1}$: 

$
Q_0=(0:0:0:1), \ Q_1=(4:-1:2:0), \ Q_2=(4:-1:-2:0), 
$

$
 Q_3=(4:1:0:0),
$

$
R_{01}=(0:0:1:0) , \ R_{02}=(4:-1:0:0), \ R_{31}=(4:-1:2i:0),
$

$
R_{32}=(4:-1:-2i:0), \ R_{21}=(4:1:0:-2\sqrt{2}i), \ R_{22}=(4:1:0:2\sqrt{2}i),
$

$
R_{11}=(4:1:0:2\sqrt{2}), \ R_{12}=(4:1:0:-2\sqrt{2})
$

\end{theorem}

\medskip

In Corollary 2 in \cite{y1} we must asume $j(E) \ne 12^3$. So we correct the corollary as follows:

\begin{corollary}
Let $\Gamma$ be an irreducible quartic curve in $\mathbb P^2$ and $E$ the normalization of it. Assume the genus of $E$ is one.  If $j(E)=12^3 \ (\mathrm{resp.}  \ne 12^3)$, then the number of Galois points is at most two \ {\rm (}resp. one{\rm )}. 
\end{corollary}

In fact, Takahashi found the curve defined by: $s^4+s^2u^2+t^4=0$. 
It is easy to see that the genus of the normalization is one and $(s:t:u)=(0:1:0)$ is a $Z_4$-point and $(1:0:0)$ is a $V_4$-point. 
By using Theorem 1, we can find many such examples as follows:  

\begin{example}
Let $L_{ij}$ and $\ell_{pq}$ be the $Z_4$ and $V_4$-lines passing though $R_{ij}$, where $0 \le i \le 3$, $j=1, 2$ and if $i=0$ or $3$ (resp. $1$ or $2$), then $(p,q)=(1,2)$ (resp. $(0,3)$).  
Let $\pi_{ij} : \mathbb P^3 \cdots \longrightarrow \mathbb P^2$ be the projection with the center $R_{ij}$.  Then,  $\pi_{ij}(C_4)=\Gamma_{ij} $ is an irreducible quartic curve and the points $\pi_{ij}(L_{ij})$ and $\pi_{ij}(\ell_{pq})$ are $Z_4$ and $V_4$-points, respectively.
For example, take the point $R=(0:0:1:0)$ as the projection center. Then, $\pi_{R}(X:Y:Z:W)=(X:Y:W)$. The $Z_4$-line $L : X=Y=0$ and $V_4$-line $\ell : X+4Y=W=0$ pass through $R$. The defining equation of $\pi_R(C_4)$ is $W^4=XY(X-4Y)^2$, $\pi_R(L)=(0:0:1)$ and $\pi_R(\ell)=(-4:1:0)$. By the projective change of coordinates 
\[
X=X'-iY', \ \ Y=-(X'+iY')/4
\]  
we get the example of Takahashi. 
\end{example}

We have an interest in the group generated by the Galois groups belonging to  Galois poits \cite{kty, mo}. In the current case we have the following: 

Let $\mathcal G_0$ (resp. $\mathcal G$) be the group generated by the Galois group belonging to $V_4$-lines ($V_4$ or $Z_4$-line)
 for $C$.  Then we have the followig.

\begin{corollary}
(1) In case $j \not=12^3$, we have $\mathcal G=\mathcal G_0=\langle\rho_0,\rho_1,\rho_2\rangle\cong Z_2\times Z_2\times Z_2$.
an example of the curve with this group is given in \cite{ky}

\[ (4y^4 + 5xy^2 -1)^2 = xy^2(x+8y^2)^2. \]

(2) In case $j=12^3$ we can show $\mathcal G=\langle\sigma_0,\sigma_2,\sigma_6\rangle$.
Putting 

\[ \alpha(z) = z+\frac12,  \ \ \beta(z)=z+\frac{3+i}4, \]

we have $\langle\alpha,\beta\rangle\cong Z_2\times Z_4$ and 

\[ \mathcal G\cong\langle\alpha,\beta\rangle \rtimes \langle\sigma_0\rangle \]

It is easy to see that $\mathcal G_0$ is a normal subgroup of $\mathcal G$.
In particular $|\mathcal G|=32$ and $\mathcal G$ is called an elliptic exceptional group $E(2,2,4)$ in \cite{ky}.
Furthermore this group appears as the group by the embedding of degree 32 of the elliptic curve $j(E)=1$.
\end{corollary}

\section{Proof}
Hereafter we treat only the case $j(E)=12^3$.  
We use the same notation and convention as in \cite{y1}. Let us recall briefly: 

\begin{enumerate}
\item[$\bullet$] $\pi : \mathbb C \longrightarrow E = \mathbb C/\mathcal L, \ \ \mathcal L=\mathbb Z + \mathbb Z i, \ \ i=\sqrt{-1}$　

\item[$\bullet$] $x=\wp(z), \ y=\wp'(z)$, \  $\wp$-functions with respect to $\mathcal L$. 

\item[$\bullet$] $\varphi : \mathbb C \longrightarrow \mathbb C/\mathcal L \xrightarrow{\sim}$ $E : y^2=4x^3-x$

\item[$\bullet$] $P_{\alpha}:=\varphi(\alpha) \in E, \ (\alpha \in \mathbb C)$, in particular, $P_0 =\varphi(0)$ 

\item[$\bullet$] $+$ denotes the sum of complex numbers $\alpha + \beta$ in $\mathbb C$ and at the same time the sum of divisors $P_{\alpha} + P_{\beta}$ on $E$

\item[$\bullet$]$\sim$ : linear equivalence 

\item[$\bullet$] Note that $P_{\alpha} + P_{\beta} \sim P_{\alpha+\beta} + P_0$ holds true.  

\item[$\bullet$] $V_4$ : Klein's four group

\item[$\bullet$] $Z_n$ : cyclic group of order $n$ 

\item[$\bullet$] $\langle \cdots \rangle$ : the group generated by $\cdots$
\end{enumerate}

Since the embedding is associated with $|4P_0|$, we can assume it is given by 

\[
f=f_4 : E \longrightarrow  \mathbb P^3, \ \ f(x,y)=(1:x^2:x:y)
\]

Put $C=f(E)$. 
The $V_4$-lines have been determined in \cite{y1}. Recall that the Galois group associated with $V_4$-line is $\langle \rho_i, \rho_j  \rangle$ for some $i,\ j$ where $0 \le i < j \le 3 $. 
Let $\sigma$ be a complex representation of a generator of the group associated with $Z_4$-line. As we see in the proof of Lemma 20 in \cite{y1}, $\sigma$ can be expressed as $\sigma(z)=iz+(m+ni)/4$, where $(m,n)=(0,0), \ (2,2),\ (3,1), \ (1,3), \ (1,1), \ (3,3), \ (2,0)$ or $(0,2)$. 
So we put as follows: 

\[
\begin{array}{lcl}
(0) \ \  \sigma_0(z)=iz & \ \ & (1) \ \sigma_1(z)=iz+{\ds \frac{1+i}{2}} \\
(2) \ \ \sigma_2(z)=iz + {\ds \frac{3+i}{4}} & \ \ & (3) \ \ \sigma_3(z)=iz+{\ds \frac{1+3i}{4}} \\
(4) \ \ \sigma_4(z)=iz+{\ds \frac{1+i}{4}} & \ \ & (5) \ \ \sigma_5(z)=iz+{\ds \frac{3+3i}{4}} \\
(6) \ \ \sigma_6(z)=iz+{\ds \frac{1}{2}} & \ \ & (7) \ \ \sigma_7(z)=iz+{\ds \frac{i}{2}}
\end{array}
\]

Furthermore we put 

\[
\rho_0(z)=-z, 
\ \ \rho_1(z)=-z+{\ds \frac{1}{2}}, 
\ \ \rho_2(z)=-z+{\ds \frac{i}{2}}, 
\ \ \rho_3(z)=-z+{\ds \frac{1+i}{2}}. 
\]

Note that 

\[
\rho_0 \equiv {\sigma_0}^2 \equiv {\sigma_1}^2 (\mathrm{mod} \mathcal L),
\ \rho_1 \equiv  {\sigma_2}^2 \equiv {\sigma_3}^2 (\mathrm{mod} \mathcal L), 
\] 

\[ \rho_2 \equiv {\sigma_4}^2 \equiv {\sigma_5}^2 (\mathrm{mod} \mathcal L), 
\ \rho_3 \equiv {\sigma_6}^2 \equiv {\sigma_7}^2 (\mathrm{mod} \mathcal L).
\]

Let $V$ be the vector space spanned by $\{1,\ x^2,\ x,\ y \}$ over $\mathbb C$. 
If $\sigma$ is an element of the Galois group associated with a Galois line $\ell$, then it induces a linear transformation $M(\sigma)$ of $V$. 
The $M(\sigma)$ defines a projective transformation, we denote it by the same letter. 
It has the following properties:
\begin{enumerate}
\item Some eigenvalue belongs to at least two independent eigenvectors. 
\item We have $M(\sigma)(\ell)=\ell$, i.e., $M(\sigma)$ induces an automorphism of $\ell \cong \mathbb P^1$.   
\end{enumerate}

There are two characterizations for the vertices, one is the following Lemma 17 in \cite{y1}: 

\begin{lemma}
There exist exactly four irreducible quadratic surfaces $S_i$ {\rm (} $0 \le i \le 3$ {\rm )} such that each $S_i$ has a singular point and contains $C$. 
Let $Q_i$ be the unique singular point of $S_i$. 
Then the four points are not coplanar. 
\end{lemma}

The other one is as follows: 

\begin{lemma}
The $M(\rho_i) \ (0 \le i \le 3)$ has two eigenvalues $\lambda_{i1}$ and $\lambda_{i2}$ which belong to one and three independent eigenvectors, respectively. 
Let $Q_i$ be the point in $\mathbb P^3$ defined by the eigenvector having the eigenvalue $\lambda_{i1}$. 
Then, these points coincide with the ones in Lemma 1. 
The line passing through $Q_i$ and $Q_j$ $(0 \le i<j \le 3)$ is a $V_4$-line. 
Four points $\{Q_1, \ Q_2,\ Q_3, \ Q_4  \}$ are not coplanar, so they form a vertex of a tetrahedron. 
\end{lemma}

\begin{proof}
These are checked by direct computations.
To find the action of $\rho_i$ on the vector space V, we can usse the action on $x=\wp(z)$ and $y=\wp'(z)$. 
Making use of the addition formula on the elliptic curve, we obtain the following.

\[
\begin{array}{lcl}
\rho_0(1,x^2,x,y) & = & (1, \ x^2, \ x, \ -y) \\
\rho_1(1,x^2,x,y) & = &  (4x^2-4x+1, \ x^2+x+\frac{1}{4}, \ 2x^2-\frac{1}{2}, \ 2y) \\
\rho_2(1,x^2,x,y) & = & (4x^2+4x+1, \ x^2-x+\frac{1}{4}, \ -2x^2+\frac{1}{2}, \ 2y) \\
\rho_3(1,x^2,x,y) & = & (4x^2, \ \frac{1}{4}, \ -x, \ -y) \\
\end{array}
\]

We obtain the following representation matrices: 

\[
\begin{array}{cc}
M(\rho_0) = 
\left(
\begin{matrix}
1&0&0&0\\
0&1&0&0\\
0&0&1&0\\
0&0&0&-1
\end{matrix}
\right), 
&
M(\rho_0) = 
\left(
\begin{matrix}
1&4&-4&0\\
1/4&1&1&0\\
-1/2&2&0&0\\
0&0&0&2
\end{matrix}
\right), 
\end{array}
\]

\[
\begin{array}{cc}
M(\rho_2) = 
\left(
\begin{matrix}
1&4&4&0\\
1/4&1&-1&0\\
1/2&-2&0&0\\
0&0&0&2
\end{matrix}
\right), 
&
M(\rho_3) = 
\left(
\begin{matrix}
0&4&0&0\\
1/4&0&0&0\\
0&0&-1&0\\
0&0&0&-1
\end{matrix}
\right)
\end{array}
\]

Therefore, the eigenvalues $\lambda$ and eigenvectors (mod constant multiplications) of $M(\rho)$ can be computed as follows :

\[
\begin{array}{lllllll}
M(\rho_0) & \lambda=-1 &:& (0,0,0,1) & \lambda=1 &:& (1,0,0,0),  \ (0,1,0,0),  \ (0,0,1,0) \\
M(\rho_1) & \lambda=-2 &:& (4,-1,2,0) & \lambda=2 &:& (1,0,-1/4,0), \  (0,1,1,0),  \ (0,0,0,1) \\
M(\rho_2) & \lambda=-2 &:& (4,-1,-2,0) & \lambda=2 &:& (4,0,1,0),  \ (0,1,-1,0),  \ (0,0,0,1) \\
M(\rho_3) & \lambda=4 &:& (4,1,0,0) & \lambda=-4 &:& (4,-1,0,0), \  (0,0,1,0), \  (0,0,0,1) \\
\end{array}
\]

\end{proof}

Similarly, we can find $Z_4$-lines by the following results : 

\[
\begin{array}{lll}
\sigma_0(1,x^2,x,y) &=& (1,x^2,-x,iy) \\
\sigma_1(1,x^2,x,y) &=& (4x^2,\frac{1}{4},x,ix) \\
\sigma_2(1,x^2,x,y) &=& (-2y+\sqrt2(i-1)x^2-\sqrt2(1+i)x-\frac{\sqrt2(i-1)}{4},\\
 && \hfil -\frac12y-\frac{\sqrt2(i-1)}{4}x^2+\frac{\sqrt2(i+1)}{4}x+\frac{\sqrt2(i-1)}{16}, \\
 && \hfil\hfil\hfil \frac{\sqrt2(1+i)}{2}x^2-\frac{\sqrt2(i-1)}{2}x-\frac{\sqrt2(1+i)}{8}, \\
 && \hfil\hfil\hfil\hfil 2x^2+\frac12) \\
\sigma_3(1,x^2,x,y) &=& (4\sqrt2iy-(1+i)(4x^2+4ix-1), \\
 && \hfil \frac14(4\sqrt2iy+(1+i)(4x^2+4ix-1)), \\
 && \hfil\hfil\hfil \frac{i-1}2(4x^2-4ix-1), \\
 && \hfil\hfil\hfil\hfil -\sqrt2 i(4x^2+1)) \\
\sigma_4(1,x^2,x,y) &=& ( -2\sqrt2(1+i)y-4ix^2-4x+i, \\
 && \hfil \frac{-1-i}{\sqrt2}y+ix^2+x-\frac i4, \\
 && \hfil\hfil 2x^2+2ix-\frac12, \\
 && \hfil\hfil\hfil -2\sqrt2(1+i)x^2-\frac{1+i}{\sqrt2} )\\
\sigma_5(1,x^2,x,y) &=& ( 2\sqrt2(1+i)y-4ix^2-4x+i, \\
 && \hfil \frac{1+i}{\sqrt2}+ix^2+x-\frac i4, \\
 && \hfil\hfil 2x^2+2ix-\frac12, \\
 && \hfil\hfil\hfil \frac{1+i}{\sqrt2}(4x^2+1))\\
\sigma_6(1,x^2,x,y) &=& ( 4x^2+4x+1, x^2-x+\frac14, 2x^2-\frac12, -2iy )\\
\sigma_7(1,x^2,x,y) &=& ( 4x^2-4x+1, x^2+x+\frac14, -2x^2+\frac12, -2iy) 
\end{array}
\]

\[
\begin{array}{cc}
M(\sigma_0) = 
\left(
\begin{matrix}
1&0&0&0\\
0&1&0&0\\
0&0&-1&0\\
0&0&0&i
\end{matrix}
\right)
&
M(\sigma_1) = 
\left(
\begin{matrix}
0&4&0&0\\
1/4&0&0&0\\
0&0&1&0\\
0&0&0&i
\end{matrix}
\right)
\end{array}
\]

\[
M(\sigma_2) = -\sqrt{2}i 
\left(
\begin{matrix}
(i+1)/4&-i-1&-i+1&-\sqrt2i\\
-(1+i)/16&(1+i)/4&(-1+i)/4&-i/2\sqrt2\\
(1-i)/8&(-1+i)/2&(1+i)/2&0\\
i/2\sqrt2&\sqrt2i&0&0
\end{matrix}
\right)
\]

\[
M(\sigma_3) = 
\left(
\begin{matrix}
1+i&-4(1+i)&4(1-i)&-4\sqrt2i\\
-(1+i)/4&1+i&-1+i&\sqrt2i\\
(1-i)/2&-2(1-i)&2(1+i)&0\\
-\sqrt2i&-4\sqrt2 i&0&0
\end{matrix}
\right)
\]

\[
M(\sigma_4) = 
\left(
\begin{matrix}
i&-4i&-4&-2\sqrt2(1+i)\\
-i/4&i&1&-(1+i)\sqrt2\\
-1/2&2&2i&0\\
-(1+i)/\sqrt2&2\sqrt2(1+i)&0&0
\end{matrix}
\right)
\]

\[
M(\sigma_5) = 
\left(
\begin{matrix}
i&-4i&-4&2\sqrt{2}(1+i)\\
-i/4&i&1&(1+i)\sqrt2\\
-1/2&2&2i&0\\
(1+i)/\sqrt2&2\sqrt2(1+i)&0&0
\end{matrix}
\right)
\]

\[
\begin{array}{cc}
M(\sigma_6) = 
\left(
\begin{matrix}
1&4&4&0\\
1/4&1&-1&0\\
-1/2&2&0&0\\
0&0&0&-2i
\end{matrix}
\right)
&
M(\sigma_7) = 
\left(
\begin{matrix}
1&4&-4&0\\
1/4&1&1&0\\
1/2&-2&0&0\\
0&0&0&-2i
\end{matrix}
\right)
\end{array}
\]

Eigenvalues $\lambda$ and eigenvectors (mod constant multiplications) of $M(\sigma)$ are as follows :

\[
\begin{array}{llllllllll}
M(\sigma_0) & \lambda=-1 &:& (0,0,0,1) \\
			& \lambda=1 &:& (1,0,0,0), (0,1,0,0) \\
            & \lambda=i &:& (0,0,0,1) \\
M(\sigma_1) & \lambda=-1 &:& (4,-1,0,0) \\
			& \lambda=1 &:& (4,1,0,0), (0,0,1,0) \\
            & \lambda=i &:& (0,0,0,1) \\
M(\sigma_2) & \lambda=\sqrt{2} &:& (4,-1,-2,0) \\
			& \lambda=-\sqrt{2} i &:& (4,0,1,\sqrt{2}i), (0,1,-1,\sqrt2i) \\
            & \lambda=\sqrt{2}i &:& (4,1,0,-2\sqrt{2}i) \\
M(\sigma_3) & \lambda=4i &:& (4,-1,-2,0) \\
			& \lambda=4 &:& (4,0,1,-2\sqrt{2}i), (0,1,-1,-\sqrt2i) \\
            & \lambda=-4 &:& (4,1,0,2\sqrt{2} i) \\
M(\sigma_4) & \lambda=-2-2i &:& (4,1,0,2\sqrt{2}) \\
            & \lambda=2+2i &:& (4,0,-1,-\sqrt{2}), (0,1,1,\sqrt2) \\
            & \lambda=-2+2i &:& (4,-1,2,0) \\
M(\sigma_5) & \lambda=-2-2i &:& (4,1,0,-2\sqrt{2}) \\
            & \lambda=2+2i &:& (4,0,-1,\sqrt{2}), (0,1,1,\sqrt2) \\
            & \lambda=-2+2i &:& (4,-1,2,0) \\
M(\sigma_6) .& \lambda=2i &:& (4,-1,2i,0) \\
            & \lambda=-2i &:& (4,-1,-2,0), (0,0,0,1) \\
            & \lambda=2 &:& (4,1,0,0) \\
M(\sigma_7) & \lambda=2i &:& (4,-1,-2i,0) \\
            & \lambda=-2i &:& (4,-1,2i,0), (0,0,0,1) \\
            & \lambda=2 &:& (4,1,0,0) \\
\end{array}
\] 

\bigskip

The proof of Corollary 2 is the same as Corollary 2 in \cite{y1}. 
It is sufficient to note the intersection points of Galois lines. 
In the case where $j(E)=12^3$, there exist points which are not the vertices $Q_i \ (0 \le i \le 3)$ but the 
intersection of $V_4$ and $Z_4$-lines. 
The projection from such points yield the curve with two Galois points.  

\section{generated Galois group}
We have studied the group generated by the Galois group belonging to Galois points \cite{kty, mo}.  
In the case of Galois embedding of elliptic curves, we have the following. 

\begin{remark}
For each Galois embedding let $\mathcal G$ be the group generated by the Galois groups belongingto the Galois subspaces. 
Then $\mathcal G$ can be realized as the Galois group for some Galois embedding of the elliptic curve. 
\end{remark}

\begin{proof}
We infer readily the theorem from Theorems 7.4 and 7.7 in \cite{ky}. 
\end{proof}

\bibliographystyle{amsplain}

\end{document}